\documentclass[12pt]{amsart} 

\usepackage{amsmath,amsthm,amstext,amssymb,amsfonts,latexsym} 
\usepackage{graphicx} 
\usepackage{xcolor} 
\usepackage[hidelinks]{hyperref} 
\usepackage{enumerate}

\allowdisplaybreaks[1] 

\newtheorem{thm}{Theorem}[section]
\newtheorem{lem}[thm]{Lemma}
\newtheorem{prop}[thm]{Proposition}
\newtheorem{cor}[thm]{Corollary}

\theoremstyle{definition} 
\newtheorem{defn}[thm]{Definition}
\theoremstyle{remark} 
\newtheorem{rem}[thm]{Remark}
\newtheorem{eg}[thm]{Example}

\newcommand{\ric}{\mathrm{Ric}}

\newcommand{\bbr}{\mathbb{R}}
\newcommand{\bbs}{\mathbb{S}}
\newcommand{\abs}[1]{\lvert #1 \rvert}
\newcommand{\owedge}{\mathbin{\bigcirc\mspace{-15mu}\wedge\mspace{3mu}}}
\newcommand{\riem}{\mathrm{Rm}}

\begin{document}

\title[Characterizations of umbilic hypersurfaces]{Characterizations of umbilic hypersurfaces in warped product manifolds}

\author{Shanze Gao}
\author{Hui Ma}

\address{Department of Mathematical Sciences, Tsinghua University, Beijing {\rm 100084}, P.R. China}
\email{shanze\_gao@163.com, ma-h@mail.tsinghua.edu.cn}

{\begin{abstract}
We consider closed orientable hypersurfaces in a wide class of warped product manifolds 
which include space forms, deSitter-Schwarzschild and Reissner-Nordstr\"{o}m manifolds. By using a new integral formula or Brendle's Heintze-Karcher type inequality, 
we present some new characterizations of umbilic hypersurfaces.
These results can be viewed as generalizations of the classical Jellet-Liebmann theorem and the Alexandrov theorem in Euclidean space.
\end{abstract}}

\thanks{This work was supported by National Natural Science Foundation of China (Grant No. 11671223 and Grant No. 11831005).}

\maketitle

\section{Introduction}

The characterization of hypersurfaces with constant mean curvature in warped product manifolds has attracted much attention recently. There are at least three types of results.
Classical Jellet-Liebmann theorem, also referred to as the Liebmann-S\"{u}ss theorem, asserts that
any closed star-shaped (or convex) immersed hypersurface in Euclidean space with constant mean curvature is a round sphere. 
This has been generalized to a class of  warped products by Montiel \cite{Montiel99}.  
Similar results are also obtained for hypersurfaces with constant higher order mean curvature or Weingarten hypersurfaces in warped products (see \cite{Alias-Impera-Rigoli13,Brendle-Eichmair,Wu-Xia}).

The classical Alexandrov theorem states that any closed embedded hypersurface of constant mean curvature in Euclidean space is a round sphere.  
This was generalized to a class of warped product manifolds by Brendle \cite{Brendle13}. 
The key step in his proof is the Minkowski type formula and a Heintze-Karcher type inequality, which 
 also works for Weingarten hypersurfaces (c.f. \cite{Brendle-Eichmair,Wu-Xia}). In \cite{Kwong-Lee-Pyo}, Kwong-Lee-Pyo proved Alexandrov type results for closed embedded hypersurfaces with radially symmetric higher order mean curvature in a class of warped products.

From a variational point of view, hypersurfaces of constant mean curvature in a Riemannian manifold are critical points of the area functional under variations preserving a certain enclosed volume (see \cite{Barbosa-doCarmo,Barbosa-doCarmo-Eschenburg}). Under the assumption of stability, constant mean curvature hypersurfaces in warped products are studied in \cite{Montiel98,Veeravalli} etc.

In this paper,  we prove Jellet-Liebmann type theorems and an Alexandrov type theorem for certain closed hypersurfaces including constant mean curvature hypersurfaces in some class of warped product manifolds.
Throughout this paper, we assume that $\bar{M}^{n+1}=[0,\bar{r})\times_{\lambda}P^{n}$ $(0<\bar{r}\leq \infty)$ is a warped product manifold endowed with a metric 
$$\bar{g}=dr^{2}+\lambda^{2}(r)g^{P},$$
where  $(P,g^{P})$ is an $n$-dimensional closed Riemannian manifold ($n\geq 2$) and $\lambda: [0,\bar{r})\rightarrow [0,+\infty)$ is a smooth positive function, called the warping function.
We first consider a hypersurface $x:M^{n}\rightarrow \bar{M}^{n+1}$ immersed in a warped product $\bar{M}^{n+1}$ whose  mean curvature $H$ satisfies 
\begin{equation}\label{Hphi}
H=\phi(r),
\end{equation}
where $\phi(r)=x^{*}(\Phi(r))$ and $\Phi(r)$ is a radially symmetric positive function on $\bar{M}$;
or
\begin{equation}\label{Halph}
H^{-\alpha}=\langle \lambda\partial_{r},\nu \rangle,
\end{equation}
where $\nu$ is a normal vector of $M$ and $\alpha>0$ is a constant.

Notice that hypersurfaces satisfying \eqref{Hphi} are critical points of the area functional under variations preserving a weighted volume (see Appendix \ref{variation}). And hypersurfaces satisfying \eqref{Halph} are self-similar solutions to the curvature flow expanding by $H^{-\alpha}$.  Both of them can be regard as generalizations of constant mean curvature hypersurfaces. 

Our first main results are Jellet-Liebmann type theorems of these hypersurfaces.

\begin{thm}\label{nablaH}
Suppose that $(\bar{M}^{n+1}, \bar{g})$ is a warped product manifold satisfying
$$\ric^{P}\geq (n-1)(\lambda'^{2}-\lambda\lambda'')g^{P},$$ 
and  $x:M\rightarrow \bar{M}$ is an immersion of a closed orientable hypersurface $M^n$ in $\bar{M}$.
If $x(M)$ is star-shaped and satisfies
\begin{equation}\label{H<0}
\langle \nabla H,\partial_{r} \rangle\leq 0,
\end{equation}
then $x(M)$ must be totally umbilic.
\end{thm}

\begin{rem}
If $M$ has constant mean curvature, Theorem \ref{nablaH} reduces to the Jellett-Liebmann type theorem proved by Montiel \cite{Montiel99}.
\end{rem}

Applying Theorem \ref{nablaH} to hypersurfaces satisfying \eqref{Hphi}, we obtain the following
\begin{cor}\label{corHphi}
Under the same assumption of Theorem \ref{nablaH}, if $x(M)$ is star-shaped and satisfies
\begin{equation*}
H=\phi(r),
\end{equation*}
where $\phi(r)=x^{*}(\Phi(r))$ and $\Phi(r)$ is a positive non-increasing function of $r$, then $x(M)$ must be totally umbilic. 
\end{cor}

\begin{rem}
An Alexandrov type theorem for the above hypersurfaces   under the embeddedness assumption was obtained by Kwong-Lee-Pyo \cite{Kwong-Lee-Pyo}.
\end{rem}

The following example shows that the non-increasing assumption on $\Phi(r)$ is necessary.
\begin{eg}\label{ellip}
Let $\Sigma$ be an ellipsoid given by
\begin{equation*}
\Sigma=\{y\in \bbr^{n+1}|y_{1}^{2}+...+y_{n}^{2}+\frac{y_{n+1}^{2}}{a^{2}}=1\}.
\end{equation*}

The mean curvature of $\Sigma$ is
\begin{equation*}
H=\frac{a}{n\sqrt{a^{2}+1-r^{2}}}(n-1+\frac{1}{a^{2}+1-r^{2}}).
\end{equation*}

It is easy to check
\begin{equation*}
\Phi=\frac{a}{n\sqrt{a^{2}+1-r^{2}}}(n-1+\frac{1}{a^{2}+1-r^{2}})
\end{equation*}
is increasing for $r$.
\end{eg}

Another application of Theorem \ref{nablaH} is about hypersurfaces satisfying \eqref{Halph}.

\begin{cor}\label{corHalph}
Under the same assumption of Theorem \ref{nablaH}, if $x(M)$ is strictly convex and satisfies
\begin{equation*}
H^{-\alpha}=\langle \lambda\partial_{r},\nu \rangle,
\end{equation*}
where $\alpha>0$ is a constant, then $x(M)$ is a slice $\{r_{0}\}\times P$ for some $r_{0}\in (0,\bar{r})$.
\end{cor}

Similar results hold for higher order mean curvature under stronger assumptions. Let $\sigma_{k}(\kappa)$ denote the $k$-th elementary symmetric polynomial of principal curvatures $\kappa=(\kappa_{1},...,\kappa_{n})$ of $x(M)$, i.e.,
\begin{equation*}
\sigma_{k}(\kappa)=\sum_{1\leq i_{1}<...<i_{k}\leq n}\kappa_{i_{1}}\cdots\kappa_{i_{k}}.
\end{equation*}
Thus, the $k$-th mean curvature is given by $H_{k}=\frac{1}{\binom{n}{k}}\sigma_{k}(\kappa)$. A hypersurface $x(M)$ is $k$-convex, if, at any point of $M$, principal curvatures $$\kappa\in\Gamma_{k}:=\{\mu\in\bbr^{n}|\sigma_{i}(\mu)>0, \text{ for }1\leq i\leq k\}.$$

\begin{thm}\label{nablaHk}
Suppose that $\bar{M}^{n+1}=[0,\bar{r})\times_{\lambda} P^{n}$ is a warped product manifold,  
where $(P,g^{P})$ is a closed Riemannian manifold with constant sectional curvature $\epsilon$ and
\begin{equation}\label{C4>=}
\frac{\lambda(r)''}{\lambda(r)}+\frac{\epsilon-\lambda(r)'^{2}}{\lambda(r)^{2}}\geq 0.
\end{equation}
Let $x:M\rightarrow \bar{M}$ be an immersion of a closed orientable hypersurface $M^n$ in $\bar{M}$.
For any fixed $k$ with $2\leq k\leq n-1$, if $x(M)$ is $k$-convex, star-shaped and satisfies
\begin{equation*}
\langle \nabla H_{k},\partial_{r} \rangle\leq 0,
\end{equation*}
then $x(M)$ must be totally umbilic.
\end{thm}

If we require that the inequality in \eqref{C4>=} is strict as in \cite{Brendle13, Brendle-Eichmair} and $H_k=\text{constant}$, we obtain 
\begin{cor}\label{corHkconst}
Suppose that $\bar{M}^{n+1}=[0,\bar{r})\times_{\lambda} P^{n}$ is a warped product manifold,  
where $(P,g^{P})$ is a closed Riemannian manifold with constant sectional curvature $\epsilon$ and
\begin{equation}\label{eq:C4}
\frac{\lambda(r)''}{\lambda(r)}+\frac{\epsilon-\lambda(r)'^{2}}{\lambda(r)^{2}}> 0.
\end{equation}
Let $x:M\rightarrow \bar{M}$ be an immersion of a closed orientable hypersurface $M^n$ in $\bar{M}$.
For any fixed $k$ with $2\leq k\leq n-1$, if $x(M)$ is $k$-convex, star-shaped and $H_k=\text{constant}$,
then $x(M)$ is a slice $\{r_0\}\times P$ for some $r_0\in (0,\bar{r})$.
\end{cor}

\begin{rem} The above corollary implies that the embeddedness condition in Theorem 2 of \cite{Brendle-Eichmair} is not necessary.
\end{rem}

\begin{cor}\label{corHkphi}
Under the same assumption of Theorem \ref{nablaHk}, if for any fixed $k$ with $2\leq k\leq n-1$, $x(M)$ is $k$-convex, star-shaped and satisfies
\begin{equation*}
H_{k}=\phi(r),
\end{equation*}
where $\phi(r)=x^{*}(\Phi(r))$ and $\Phi(r)$ is a positive non-increasing function of $r$, then $x(M)$ must be totally umbilic. 
\end{cor}

\begin{cor}\label{corHkalph}
Under the same assumption of Theorem \ref{nablaHk}, if for any fixed $k$ with $2\leq k\leq n-1$, $x(M)$ is strictly convex and satisfies
\begin{equation*}
H_{k}^{-\alpha}=\langle \lambda\partial_{r},\nu \rangle,
\end{equation*}
where $\alpha>0$ is a constant, then $x(M)$ is a slice $\{r_{0}\}\times P$ for some $r_{0}\in (0,\bar{r})$. 
\end{cor}

Now we turn to the warped product manifold $\bar{M}^{n+1}=[0,\bar{r})\times_{\lambda} P^{n}$, where  $(P,g^{P})$ is with a closed Riemannian manifold with constant sectional curvature $\epsilon$.
As in \cite{Brendle13,Wu-Xia}, we list four conditions of the warping function $\lambda:[0,\bar{r})\rightarrow [0,+\infty)$:
\begin{itemize}
\item[(C1)] $\lambda'(0)=0$ and $\lambda''(0)>0$.
\item[(C2)] $\lambda'(r)>0$ for all $r\in(0,\bar{r})$.
\item[(C3)] The function $$2\frac{\lambda''(r)}{\lambda(r)}-(n-1)\frac{\epsilon-\lambda'(r)^{2}}{\lambda(r)^{2}}$$
is non-decreasing for $r\in(0,\bar{r})$.
\item[(C4)] We have $$\frac{\lambda''(r)}{\lambda(r)}+\frac{\epsilon-\lambda'(r)^{2}}{\lambda(r)^{2}}>0$$
for all $r\in(0,\bar{r})$.
\end{itemize}

Instead of star-sharpness or convexity, under the embeddedness assumption, we study hypersurfaces satisfying 
\begin{equation}\label{uHlambda}
H_{k}^{-\alpha}\lambda'=\langle \lambda\partial_{r},\nu \rangle,
\end{equation}
and prove the following Alexandrov type theorem.

\begin{thm}\label{Hklambda}
Suppose that $(\bar{M}, \bar{g})$ is a warped product manifold satisfying conditions (C1)-(C4).
Let $x:M\rightarrow \bar{M}$ be an immersion of a connected closed embedded orientable hypersurface $M^n$ in $\bar{M}$.
If $H_{k}>0$ and $x(M)$ satisfies
\begin{equation}
H_{k}^{-\alpha}\lambda'=\langle \lambda\partial_{r},\nu \rangle,
\end{equation}
for any fixed $k$ with $1\leq k\leq n$ and $\alpha\geq \frac{1}{k}$, then $x(M)$ is a slice $\{r_{0}\}\times P$ for some $r_{0}\in (0,\bar{r})$.
\end{thm}

\begin{rem}
It is interesting to compare Theorem \ref{Hklambda} and Corollary \ref{corHkalph} in the special case when
$P=\bbs^{n}$ and $\lambda(r)=r$, i.e. $\bar{M}$ is Euclidean space $\bbr^{n+1}$. With embeddedness and less convexity requirement (only $H_{k}>0$), Theorem \ref{Hklambda} leads to the conclusion including the case when $k=n$.
\end{rem}

\begin{rem}
Throughout the paper, the assumptions for the ambient spaces $\bar{M}$ are satisfied by space forms, the deSitter-Schwarzschild, the Reissner-Nordstr\"{o}m manifolds and many other manifolds (c.f. \cite{Brendle13}).
\end{rem}

The paper is organized as follows. In Section \ref{wprod}, we list some useful properties of warped products. In Section \ref{intformula}, we derive an integral formula which is the key to the proof of our main theorems. In Section \ref{proof}, we present the proofs of Theorem \ref{nablaH}, Theorem \ref{nablaHk} and the corollaries. In Section \ref{embed}, we prove Theorem \ref{Hklambda}. In Appendix \ref{variation}, we show that a hypersurface with a given positive mean curvature function is the critical point of the area functional under variations preserving weighted volume.
Throughout the paper, the summation convention is used unless otherwise stated.

\section{Preliminaries of warped products}
\label{wprod}

In this section, we list some basic properties of warped products $(\bar{M}=[0,\bar{r})\times_{\lambda}P^{n}, \bar{g})$ given above (see \cite{ONeill}). 

\begin{prop}\label{connect}
Suppose $U,V\in \Gamma(TP)$. The Levi-Civita connection $\bar{\nabla}$ of a warped product $(\bar{M}=[0,\bar{r})\times_{\lambda} P, \bar{g})$ satisfies
\begin{enumerate}[i)]
\item $\bar{\nabla}_{\partial_{r}}\partial_{r}=0$,
\item $\bar{\nabla}_{\partial_{r}}V=\bar{\nabla}_{V}\partial_{r}=\frac{\lambda'}{\lambda}V$,
\item $\bar{\nabla}_{V}U=\nabla^{P}_{V}U-\frac{\lambda'}{\lambda}\bar{g}(V,U)\partial_{r}$,
\end{enumerate}
where $\nabla^P$ is the Levi-Civita connection of $(P, g^P)$.
\end{prop}

\begin{rem}
From the preceding proposition, we know that any slice $\{r\}\times P$ in a warped product $\bar{M}=[0,\bar{r})\times_{\lambda} P$ is totally umbilic.
\end{rem}

\begin{prop}\label{curv4}
Suppose $Y_{1},Y_{2},Y_{3},Y_{4}\in \Gamma(TP)$. The $(0,4)$-Riemannian curvature tensor $\overline{\riem}$ of a warped product $(\bar{M}=[0,\bar{r})\times_{\lambda} P, \bar{g})$ satisfies
\begin{enumerate}[i)]
\item $\overline{\riem}(\partial_{r},Y_{1},\partial_{r},Y_{2})=-\frac{\lambda''}{\lambda}\bar{g}(Y_{1},Y_{2})$,
\item $\overline{\riem}(\partial_{r},Y_{1},Y_{2},Y_{3})=0$,
\item $\overline{\riem}(Y_{1},Y_{2},Y_{3},Y_{4})=\lambda^{2}\riem^{P}(Y_{1},Y_{2},Y_{3},Y_{4})-\frac{\lambda'^{2}}{\lambda^{2}}(\bar{g}(Y_{1},Y_{3})\bar{g}(Y_{2},Y_{4})
-\bar{g}(Y_{2},Y_{3})\bar{g}(Y_{1},Y_{4})),
$
\end{enumerate}
where $\riem^{P}$ is the $(0,4)$-Riemannian curvature tensor  of $(P, g^P)$.
\end{prop}

\begin{prop}\label{ric}
Suppose $U,V\in \Gamma(TP)$. The Ricci curvature tensor $\overline{\ric}$ of a warped product $(\bar{M}=[0,\bar{r})\times_{\lambda} P, \bar{g})$ satisfies
\begin{enumerate}[i)]
\item $\overline{\ric}(\partial_{r},\partial_{r})=-n\frac{\lambda''}{\lambda},$
\item $\overline{\ric}(\partial_{r},V)=0,$
\item $\overline{\ric}(V,U)=\ric^{P}(V,U)-\left(\frac{\lambda''}{\lambda}+(n-1)\frac{\lambda'^{2}}{\lambda^{2}}\right)\bar{g}(V,U),$
\end{enumerate}
where $\ric^{P}$ is the Ricci curvature tensor  of $(P, g^P)$.
\end{prop}

For the convenience, we introduce the Kulkarni-Nomizu product  $\owedge$. For any two $(0,2)$-type symmetric tensors $h$ and $w$, $h\owedge w$ is the $4$-tensor given by
\begin{align*}
(h\owedge w)(X_{1},X_{2},X_{3},X_{4})
&=h(X_{1},X_{3})w(X_{2},X_{4})+h(X_{2},X_{4})w(X_{1},X_{3})\\
&\quad-h(X_{1},X_{4})w(X_{2},X_{3})-h(X_{2},X_{3})w(X_{1},X_{4}).
\end{align*}
The following result is a corollary of Proposition \ref{curv4} through a straightforward calculation and 
the proof is given for the completeness.

\begin{prop}
Suppose $(P,g^{P})$ is a Riemannian manifold with constant sectional curvature $\epsilon$. The Riemannian curvature tensor $\overline{\riem}$ of a warped product $\bar{M}=[0,\bar{r})\times_{\lambda} P$ can be expressed as follows:
\begin{equation}\label{eq:Rm}
\overline{\riem}=\frac{\epsilon-\lambda'^{2}}{2\lambda^{2}}\bar{g}\owedge\bar{g}-\left(\frac{\lambda''}{\lambda}+\frac{\epsilon-\lambda'^{2}}{\lambda^{2}}\right)\bar{g}\owedge dr^{2}.
\end{equation}
\end{prop}

\begin{proof}

Let $e_{A},e_{B},e_{C},e_{D}\in \Gamma(T\bar{M})$, $r_{A}=\bar{g}(e_{A},\partial_{r})$ and $e_{A}^{*}=e_{A}-r_{A}\partial_{r}$. Using Proposition \ref{curv4}, we have
\begin{align*}
&\overline{\riem}(e_{A},e_{B},e_{C},e_{D})\\
&=\overline{\riem}(e_{A}^{*},e_{B}^{*},e_{C}^{*},e_{D}^{*})+\overline{\riem}(e_{A}^{*},r_{B}\partial_{r},e_{C}^{*},r_{D}\partial_{r})+\overline{\riem}(e_{A}^{*},r_{B}\partial_{r},r_{C}\partial_{r},e_{D}^{*})\\
&\quad+\overline{\riem}(r_{A}\partial_{r},e_{B}^{*},e_{C}^{*},r_{D}\partial_{r})+\overline{\riem}(r_{A}\partial_{r},e_{B}^{*},r_{C}\partial_{r},e_{D}^{*})\\
&=\lambda^{2}\riem^{P}(e_{A}^{*},e_{B}^{*},e_{C}^{*},e_{D}^{*})-\frac{\lambda'^{2}}{\lambda^{2}}\Big(\bar{g}(e_{A}^{*},e_{C}^{*})\bar{g}(e_{B}^{*},e_{D}^{*})-\bar{g}(e_{B}^{*},e_{C}^{*})\bar{g}(e_{A}^{*},e_{D}^{*})\Big)\\
&\quad-\frac{\lambda''}{\lambda}\Big(r_{B}r_{D}\bar{g}(e_{A}^{*},e_{C}^{*})-r_{B}r_{C}\bar{g}(e_{A}^{*},e_{D}^{*})-r_{A}r_{D}\bar{g}(e_{B}^{*},e_{C}^{*})+r_{A}r_{C}\bar{g}(e_{B}^{*},e_{D}^{*})\Big).
\end{align*}

Since
\begin{align*}
\bar{g}(e_{A}^{*},e_{C}^{*})=\bar{g}(e_{A},e_{C})-r_{A}r_{C}
=(\bar{g}-dr^{2})(e_{A},e_{C}),
\end{align*}
we know
\begin{align*}
& \bar{g}(e_{A}^{*},e_{C}^{*})\bar{g}(e_{B}^{*},e_{D}^{*})-\bar{g}(e_{B}^{*},e_{C}^{*})\bar{g}(e_{A}^{*},e_{D}^{*})\\
&\qquad=\frac{1}{2}(\bar{g}-dr^{2})\owedge (\bar{g}-dr^{2})(e_{A},e_{B},e_{C},e_{D})
\end{align*}
and
\begin{align*}
& r_{B}r_{D}\bar{g}(e_{A}^{*},e_{C}^{*})-r_{B}r_{C}\bar{g}(e_{A}^{*},e_{D}^{*})-r_{A}r_{D}\bar{g}(e_{B}^{*},e_{C}^{*})+r_{A}r_{C}\bar{g}(e_{B}^{*},e_{D}^{*})\\
&\qquad=(\bar{g}-dr^{2})\owedge dr^{2}(e_{A},e_{B},e_{C},e_{D}).
\end{align*}

Using the sectional curvatures of $(P,g^{P})$ is a constant $\epsilon$, i.e., $\riem^{P}=\frac{\epsilon}{2}g^{P}\owedge g^{P}$, we have
\begin{align}
\lambda^{2}\riem^{P}(e_{A}^{*},e_{B}^{*},e_{C}^{*},e_{D}^{*})=\frac{\epsilon}{2\lambda^{2}}(\bar{g}-dr^{2})\owedge(\bar{g}-dr^{2})(e_{A},e_{B},e_{C},e_{D}).
\end{align}

Combining these together, we obtain
\begin{align*}
\overline{\riem}=\frac{\epsilon-\lambda'^{2}}{2\lambda^{2}}\bar{g}\owedge \bar{g}-\left(\frac{\epsilon-\lambda'^{2}}{\lambda^{2}}+\frac{\lambda''}{\lambda}\right)\bar{g}\owedge dr^{2}.
\end{align*}
\end{proof}

\section{An integral formula}
\label{intformula}

In this section, we obtain an integral formula by the divergence theorem, which is the key to the proof of our main results.

Let $x:M\rightarrow \bar{M}$ be an immersion of a closed orientable hypersurface $M^n$ into a warped product $\bar{M}^{n+1}=[0,\bar{r})\times_{\lambda} P^{n}$ endowed with a metric $\bar{g}=dr^{2}+\lambda^{2}(r)g^{P}$. Let $\nu$ be a normal vector field of $M$ and $h=(h_{ij})$ denote the second fundamental form with respect to an orthogonal frame $\{e_{1},...,e_{n}\}$ on $M$ defined by
 $$h_{ij}=\langle \nabla_{i}x,\nabla_{j}\nu \rangle.$$
The principal curvatures $\kappa=(\kappa_{1},...,\kappa_{n})$ are the eigenvalues of  $h$. 
Thus the $k$-th elementary symmetric polynomials of principal curvatures can be expressed as follows
\begin{equation*}
\sigma_{k}(\kappa(h))=\frac{1}{k!}\delta_{j_{1}...j_{k}}^{i_{1}...i_{k}}h_{i_{1}j_{1}}\cdots h_{i_{k}j_{k}},
\end{equation*}
where $\delta_{j_{1}...j_{k}}^{i_{1}...i_{n}}$ is the generalized Kronecker symbol.
Let $\sigma_{k;i}(\kappa)$ denote $\sigma_k(\kappa)$ with $\kappa_{i}=0$ and $\sigma_{k;ij}(\kappa)$,  with $i\neq j$, denote the symmetric function $\sigma_k(\kappa)$ with $\kappa_i=\kappa_j=0$.

The following proposition is from a standard calculation (see also \cite{Brendle-Eichmair, Kwong-Lee-Pyo}) and the proof is given for the completeness.

\begin{prop}\label{div}
Under an orthonormal frame such that $h_{ij}=\kappa_{i}\delta_{ij}$, we have the following equality
\begin{equation*}
\sum_{i}\nabla_{i}(\frac{\partial \sigma_{k}(h)}{\partial h_{ij}})=-\sum_{p\neq j}\bar{R}_{\nu pjp}\sigma_{k-2;jp}(\kappa)
\end{equation*}
for any fixed $j$ and $2\leq k\leq n$.
\end{prop}

\begin{proof}
Let $\tilde{h}=I+th$. Then,
\begin{equation}
\sigma_{n}(\tilde{h})=\sigma_{n}(I+th)=\sum_{k=0}^{n}t^{k}\sigma_{k}(h).
\end{equation}

Using
\begin{equation*}
\frac{\partial\sigma_{n}(\tilde{h})}{\partial h_{ij}}=t(\tilde{h}^{-1})_{ij}\sigma_{n}(\tilde{h}),
\end{equation*}
\begin{equation*}
\sum_{i=1}^{n}\nabla_{i}(\tilde{h}^{-1})_{ij}=-t(\tilde{h}^{-1})_{ip}(\tilde{h}^{-1})_{qj}\nabla_{i}h_{pq},
\end{equation*}
for arbitrary $t$ and the Codazzi equation
 $$\nabla_{i}h_{pq}=\nabla_{q}h_{pi}+\bar{R}_{\nu pqi},$$
 we have
\begin{equation}\label{tk}
\begin{aligned}
&t^{k}\nabla_{i}(\frac{\partial \sigma_{k}(h)}{\partial h_{ij}})=\nabla_{i}(\frac{\partial\sigma_{n}(\tilde{h})}{\partial h_{ij}})\\
&\qquad=t^{2}\sigma_{n}(\tilde{h})\left(-(\tilde{h}^{-1})_{ip}(\tilde{h}^{-1})_{qj}\nabla_{i}h_{pq}+(\tilde{h}^{-1})_{ij}(h^{-1})_{pq}\nabla_{i}h_{pq}\right)\\
&\qquad=t^{2}(\tilde{h}^{-1})_{pq}(\tilde{h}^{-1})_{ij}\sigma_{n}(\tilde{h})\left(-\nabla_{q}h_{pi}+\nabla_{i}h_{pq}\right)\\
&\qquad=t^{2}(\tilde{h}^{-1})_{pq}(\tilde{h}^{-1})_{ij}\sigma_{n}(\tilde{h})\bar{R}_{\nu pqi}.
\end{aligned}
\end{equation}

Now we choose a local orthonormal frame $\{e_1, \cdots, e_n\}$ such that $h_{ij}=\kappa_{i}\delta_{ij}$. Then
\begin{equation*}
(\tilde{h}^{-1})_{ij}=\frac{\delta_{ij}}{1+t\kappa_{i}}.
\end{equation*}

Thus,
\begin{align*}
&(\tilde{h}^{-1})_{pq}(\tilde{h}^{-1})_{ij}\sigma_{n}(\tilde{h})\bar{R}_{\nu pqi}=-\frac{\bar{R}_{\nu pjp}}{(1+t\kappa_{j})(1+t\kappa_{p})}\prod_{l=1}^{n}(1+t\kappa_{l})\\
&\qquad=-\bar{R}_{\nu pjp}\prod_{l\in\{1,...,n\}\backslash\{j,p\}}(1+t\kappa_{l})\\
&\qquad=-\bar{R}_{\nu pjp}\sum_{k=2}^{n}t^{k-2}\sigma_{k-2;jp}(\kappa).
\end{align*}

Combining with \eqref{tk}, we have
\begin{equation*}
\sum_{k=1}^{n}t^{k}\nabla_{i}(\frac{\partial \sigma_{k}(h)}{\partial h_{ij}})=-\bar{R}_{\nu pjp}\sum_{k=2}^{n}t^{k}\sigma_{k-2;jp}(\kappa).
\end{equation*}

Comparing the coefficients of $t^{k}$, we have
\begin{equation}
\sum_{i}\nabla_{i}(\frac{\partial \sigma_{k}(h)}{\partial h_{ij}})=-\sum_{p\neq j}\bar{R}_{\nu pjp}\sigma_{k-2;jp}(\kappa)
\end{equation}
for each $k\in\{2,...,n\}$.
\end{proof}

Denote $\eta=x^{*}(\int_{0}^{r}\lambda(s)ds)$ and $u=\langle \lambda\partial_{r},\nu \rangle$. We have the following integral formula.

\begin{lem}\label{key}
Suppose $x(M)$ is a closed hypersurface of $\bar{M}$. The following equality holds
\begin{align*}
\int_{M} \{&-(n-k)\langle \nabla\sigma_{k},\lambda\partial_{r} \rangle+((n-k)\sigma_{1}\sigma_{k}-n(k+1)\sigma_{k+1})u\\
&-n\bar{R}_{\nu pjp}\langle \lambda\partial_{r},e_{j} \rangle\sigma_{k-1;jp}\}d\mu=0.
\end{align*}
\end{lem}

\begin{proof}
From a straightforward calculation, we have
\begin{align*}
&\nabla_{i}(k\sigma_{k}\nabla_{i}\eta-n\frac{\partial\sigma_{k}}{\partial h_{ij}}\nabla_{j}u)=k\langle \nabla\sigma_{k},\nabla\eta \rangle+k\sigma_{k}\Delta\eta-n\nabla_{i}(\frac{\partial\sigma_{k}}{\partial h_{ij}})\nabla_{j}u-n\frac{\partial\sigma_{k}}{\partial h_{ij}}\nabla_{i}\nabla_{j}u\\
&\qquad=k\langle \nabla\sigma_{k},\nabla\eta \rangle+k\sigma_{k}(n\lambda'-\sigma_{1}u)-n\nabla_{i}(\frac{\partial\sigma_{k}}{\partial h_{ij}})\langle \lambda\partial_{r},h_{jl}e_{l} \rangle\\
&\qquad\quad+n\frac{\partial\sigma_{k}}{\partial h_{ij}}(-\lambda'h_{ij}-\langle \lambda\partial_{r},h_{jli}e_{l} \rangle+h_{jl}h_{li}u)\\
&\qquad=k\langle \nabla\sigma_{k},\lambda\partial_{r} \rangle+(n-k)\sigma_{k}\sigma_{1}u-n(k+1)\sigma_{k+1}u-n\nabla_{i}(\frac{\partial\sigma_{k}}{\partial h_{ij}})\langle \lambda\partial_{r},h_{jl}e_{l} \rangle\\
&\qquad\quad-n\langle \lambda\partial_{r},\nabla \sigma_{k} \rangle-n\frac{\partial\sigma_{k}}{\partial h_{ij}}\langle \lambda\partial_{r},e_{l} \rangle\bar{R}_{\nu jli}.
\end{align*}

Using Proposition \ref{div}, we know that
\begin{align*}
&-n\nabla_{i}(\frac{\partial\sigma_{k}}{\partial h_{ij}})\langle \lambda\partial_{r},h_{jl}e_{l} \rangle-n\frac{\partial\sigma_{k}}{\partial h_{ij}}\langle \lambda\partial_{r},e_{l} \rangle\bar{R}_{\nu jli}\\
&\qquad=n\bar{R}_{\nu pjp}\sigma_{k-2;jp}\langle \lambda\partial_{r},\kappa_{j}e_{j} \rangle-n\sigma_{k-1;i}\langle \lambda\partial_{r},e_{l} \rangle\bar{R}_{\nu ili}\\
&\qquad=n\bar{R}_{\nu pjp}\langle \lambda\partial_{r},e_{j} \rangle(\sigma_{k-2;jp}\kappa_{j}-\sigma_{k-1;p})\\
&\qquad=-n\bar{R}_{\nu pjp}\langle \lambda\partial_{r},e_{j} \rangle\sigma_{k-1;jp}.
\end{align*}

Combining these equalities and using divergence theorem, we finish the proof.
\end{proof}

\section{Proof of the main theorems}
\label{proof}

\begin{proof}[Proof of Theorem \ref{nablaH}]
Using Lemma \ref{key} for $k=1$, we know that
\begin{equation}\label{k=1}
\int_{M}\{-n(n-1)\langle \nabla H,\lambda\partial_{r} \rangle+((n-1)\sigma_{1}^{2}-2n\sigma_{2})u-n\overline{\ric}(\nu,\lambda\partial_{r}^{\top})\}d\mu=0,
\end{equation}
where $\partial_{r}^{\top}$ deontes the tangent part of $\partial_{r}$.

From Newton inequality and $x(M)$ is star-shaped ($\langle \partial_{r},\nu \rangle >0$), we have
\begin{equation}\label{sig2}
((n-1)\sigma_{1}^{2}-2n\sigma_{2})u\geq 0.
\end{equation}
And the equality of \eqref{sig2} occurs if and only $\kappa_{1}=...=\kappa_{n}$.

Let $\nu^{P}=\nu-\langle \nu,\partial_{r} \rangle\partial_{r}$. Since
\begin{align*}
&\overline{\ric}(\nu,\lambda\partial_{r}^{\top})=\overline{\ric}(\nu,\lambda\partial_{r})-\overline{\ric}(\nu,\nu)u\\
&\qquad=u\left(-n\frac{\lambda''}{\lambda}+n\frac{\lambda''}{\lambda}\langle \partial_{r},\nu \rangle^{2}-\ric^{P}(\nu^{P},\nu^{P})+\left(\frac{\lambda''}{\lambda}+(n-1)\frac{\lambda'^{2}}{\lambda^{2}}\right)\abs{\nu^{P}}^{2}\right)\\
&\qquad=-u\left(\ric^{P}(\nu^{P},\nu^{P})+(n-1)\left(\frac{\lambda''}{\lambda}-\frac{\lambda'^{2}}{\lambda^{2}}\right)\abs{\nu^{P}}^{2}\right)\\
&\qquad=-u\left(\ric^{P}(\nu^{P},\nu^{P})+(n-1)\left(\lambda\lambda''-\lambda'^{2}\right)g^{P}(\nu^{P},\nu^{P})\right),
\end{align*}
we know $\overline{\ric}(\nu,\lambda\partial_{r}^{\top})\leq 0$ by assumption.

Combining these estimates with $\langle \nabla H,\partial_{r} \rangle\leq 0$, we obtain the left hand side \eqref{k=1} is nonnegative. These implies the inequalities are  actually equalities at any point of $M$. Thus, $x(M)$ is totally umbilic.
\end{proof}

\begin{rem}
In the previous proof, if $\ric^{P}>(n-1)(\lambda'^{2}-\lambda\lambda'')g^{P}$, we also obtain $\partial_{r}^{\top}=0$. This means $\partial_{r}$ is the normal vector of $x(M)$ which implies $x(M)$ is a slice.
\end{rem}

\begin{proof}[Proof of Theorem \ref{nablaHk}]
Under the assumption, it follows from \eqref{eq:Rm} that
\begin{equation*}
\bar{R}_{\nu pjp}=-\left(\frac{\lambda''}{\lambda}+\frac{\epsilon-\lambda'^{2}}{\lambda^{2}}\right)\langle \partial_{r},\nu \rangle \langle \partial_{r},e_{j} \rangle,
\end{equation*}
for any fixed $p$ and $j\neq p$. Using 
\begin{equation*}
\sum_{p\neq j}\sigma_{k-1;jp}=(n-k)\sigma_{k-1;j},
\end{equation*}
we know
\begin{align*}
&-\bar{R}_{\nu pjp}\langle \lambda\partial_{r},e_{j} \rangle\sigma_{k-1;jp}=u\left(\frac{\lambda''}{\lambda}+\frac{\epsilon-\lambda'^{2}}{\lambda^{2}}\right)\langle \partial_{r},e_{j} \rangle^{2}\sigma_{k-1;jp}\\
&=(n-k)u\left(\frac{\lambda''}{\lambda}+\frac{\epsilon-\lambda'^{2}}{\lambda^{2}}\right)\langle \partial_{r},e_{j} \rangle^{2}\sigma_{k-1;j}\geq 0,
\end{align*}
where the inequality follows from that $\frac{\lambda''}{\lambda}+\frac{\epsilon-\lambda'^{2}}{\lambda^{2}}\geq 0$, $x(M)$ is star-shaped and $k$-convex.

Similar to the previous proof, we know
\begin{equation*}
-(n-k)\langle \nabla\sigma_{k},\lambda\partial_{r} \rangle\geq 0
\end{equation*}
and
\begin{equation*}
((n-k)\sigma_{1}\sigma_{k}-n(k+1)\sigma_{k+1})u\geq 0.
\end{equation*}
But Lemma \ref{key} shows that the integral of these terms are zero. Thus, we know all these inequalities are actually equalities which implies $x(M)$ is totally umbilic. The rest of the proof is similar to the previous one.
\end{proof}

\begin{proof}[Proof of Corollary \ref{corHkconst}]
Since $H_{k}=\text{constant}$, we know $\nabla\sigma_{k}=0$. From the proof of Theorem \ref{nablaHk},
\begin{equation*}
-\bar{R}_{\nu pjp}\langle \lambda\partial_{r},e_{j} \rangle\sigma_{k-1;jp}=0.
\end{equation*}
Combining with the condition \eqref{eq:C4}, we obtain $\abs{\partial_{r}^{\top}}=0$, which implies $x(M)$ is a slice.
\end{proof}

\begin{proof}[Proof of Corollary \ref{corHphi} and Corollary \ref{corHkphi}]

From $H_{k}=\phi(r)$, by direct calculation, we have
\begin{equation*}
\langle \nabla H_{k},\partial_{r} \rangle=\phi'\abs{\partial_{r}^{\top}}^{2}.
\end{equation*}

Since $\Phi(r)$ is non-increasing, we know $\phi'(r)\leq 0$. Thus, $$\langle \nabla H_{k},\partial_{r} \rangle\leq 0.$$

By Theorem \ref{nablaH} or Theorem \ref{nablaHk}, we finish the proof.
\end{proof}

\begin{proof}[Proof of Corollary \ref{corHalph} and Corollary \ref{corHkalph}]

From $H_{k}^{-\alpha}=u$, we have
\begin{align*}
\langle \nabla H_{k},\partial_{r} \rangle=-\frac{1}{\alpha}u^{-\frac{1}{\alpha}-1}\langle \nabla u,\partial_{r} \rangle=-\frac{1}{\alpha}u^{-\frac{1}{\alpha}-1}\lambda\kappa_{i}\langle e_{i},\partial_{r} \rangle^{2}.
\end{align*}

Since $x(M)$ is strictly convex, from $u=H_{k}^{-\alpha}$, we know $u>0$. By $\alpha>0$, $$\langle \nabla H_{k},\partial_{r} \rangle\leq 0.$$

From the proof of Theorem \ref{nablaH} or Theorem \ref{nablaHk}, we know $$\langle \nabla H_{k},\partial_{r} \rangle=0.$$
This implies $\partial_{r}$ is the normal vector of $x(M)$, which means $x(M)$ is a slice.
 
\end{proof}

\section{Proof of Theorem \ref{Hklambda}}
\label{embed}

In this section we give the proof of Theorem \ref{Hklambda}.
By Lemma 2.3 in \cite{Li-Wei-Xiong}, we know $x(M)$ is $k$-convex from $H_{k}>0$. Thus, Maclaurin's inequality 
\begin{equation}\label{Mac}
H_{k}^{\frac{1}{k}}\leq H_{k-1}^{\frac{1}{k-1}}
\end{equation}
holds.
From Brendle's Heintze-Karcher type inequality established in \cite{Brendle13}
\begin{align*}
\int_{M}ud\mu \leq \int_{M}\frac{\lambda'}{H_{1}}d\mu
\end{align*}
and the Minkowski type formula (see \cite{Brendle-Eichmair})
\begin{align*}
\int_{M}H_{k}ud\mu\geq \int_{M}H_{k-1}\lambda'd\mu,
\end{align*}
combining with Maclaurin's inequality and $H_{k}^{-\alpha}\lambda'=u$, we obtain
\begin{equation}\label{Hkalf}
\int_{M}H_{k}^{-\alpha}\lambda'd\mu \leq \int_{M}\frac{\lambda'}{H_{1}}d\mu \leq \int_{M}H_{k}^{-\frac{1}{k}}\lambda'd\mu
\end{equation}
and
\begin{equation}\label{Min}
\int_{M}H_{k}^{1-\alpha}\lambda'd\mu\geq \int_{M}H_{k-1}\lambda'd\mu.
\end{equation}

By H\"{o}lder's inequality, Maclaurin's inequality \eqref{Mac} and \eqref{Min}, we have
\begin{align*}
\int_{M}H_{k}^{-\frac{1}{k}}\lambda'd\mu &\leq \left(\int_{M}H_{k-1}^{1-p}H_{k}^{-\frac{p}{k}}\lambda'd\mu\right)^{\frac{1}{p}}\left(\int_{M}H_{k-1}\lambda'd\mu\right)^{\frac{p-1}{p}}\\
&\leq \left(\int_{M}H_{k}^{\frac{-kp-1+k}{k}}\lambda'd\mu\right)^{\frac{1}{p}}\left(\int_{M}H_{k}^{1-\alpha}\lambda'd\mu\right)^{\frac{p-1}{p}}.
\end{align*}

Choose $p$ such that $p-1+\frac{1}{k}=\alpha$, then $p=\frac{k\alpha+k-1}{k}$. Notice that $p\geq 1$ implies $\alpha\geq \frac{1}{k}$.
The above inequality becomes
\begin{equation}\label{Hk}
\int_{M}H_{k}^{-\frac{1}{k}}\lambda'd\mu \leq \left(\int_{M}H_{k}^{-\alpha}\lambda'd\mu\right)^{\frac{k}{k\alpha+k-1}}\left(\int_{M}H_{k}^{1-\alpha}\lambda'd\mu\right)^{\frac{k\alpha-1}{k\alpha+k-1}}.
\end{equation}

Using H\"{o}lder's inequality, \eqref{Mac} and \eqref{Min} again as before, we obtain
\begin{align*}
\int_{M}H_{k}^{1-\alpha}\lambda'd\mu &\leq \left(\int_{M}H_{k}^{\frac{-1+k+p-pk\alpha}{k}}\lambda'd\mu\right)^{\frac{1}{p}}\left(\int_{M}H_{k}^{1-\alpha}\lambda'd\mu\right)^{\frac{p-1}{p}}.
\end{align*}
Equivalently,
\begin{align*}
\int_{M}H_{k}^{1-\alpha}\lambda'd\mu &\leq \int_{M}H_{k}^{\frac{-1+k+p-pk\alpha}{k}}\lambda'd\mu.
\end{align*}

Now we choose $p$ such that $\frac{-1+k+p-pk\alpha}{k}=-\alpha$. Then $p=\frac{k\alpha-1+k}{k\alpha-1}$. Thus, the above inequality is
\begin{align*}
\int_{M}H_{k}^{1-\alpha}\lambda'd\mu &\leq \int_{M}H_{k}^{-\alpha}\lambda'd\mu.
\end{align*}

Substituting the above inequality into \eqref{Hk}, we obtain
\begin{equation*}
\int_{M}H_{k}^{-\frac{1}{k}}\lambda'd\mu \leq \int_{M}H_{k}^{-\alpha}\lambda'd\mu.
\end{equation*}

Combining the above inequality with \eqref{Hkalf}, we know that the equality of the Heintze-Karcher type inequality occurs. As \cite{Brendle13}, we finish the proof.


\appendix

\section{}
\label{variation}

Let $ \bar{M}^{n+1}$ be an oriented Riemannian manifold and let $x:M^{n}\rightarrow \bar{M}^{n+1}$ be an immersion of a closed smooth $n$-dimensional manifold $M$ into $\bar{M}^{n+1}$. 
Suppose that a smooth map $X:(-\epsilon,\epsilon)\times M\rightarrow \bar{M}$ is a normal variation satisfying
\begin{equation*}
\frac{\partial}{\partial t}X=-f\nu,
\end{equation*}
where $f$ is a smooth function on $M$ and $\nu$ is the unit normal of $X(t, M)$.

We introduce the weighted volume $V: (-\epsilon, \epsilon)\rightarrow \mathbb{R}$ by
\begin{equation}
V(t)=\int_{[0,t]\times M}X^{*}(e^{\Psi}d\bar{\mu}),
\end{equation}
where $d\bar{\mu}$ is a standard volume element of $\bar{M}$ and $\Psi$ is a smooth function on $\bar{M}$.
Thus $V(t)$ represents the (oriented) weighted volume sweeping by $M$ on the time interval $[0,t)$.
By the same calculations as in \cite{Barbosa-doCarmo-Eschenburg}, we have
\begin{equation}
V'(t)=\int_{M}fX^{*}(e^{\Psi})d\mu,
\end{equation}
where $d\mu$ is the volume element of $M$ with respect to the induced metric.

\begin{defn}
A variation of $x:M\rightarrow\bar{M}$ is called weighted volume-preserving variation if $V(t)\equiv0$.
\end{defn}

Denote
\begin{equation}
J(t)=A(t)+nH_{0}V(t),
\end{equation}
where $A(t)=\int_{M}d\mu$ and $H_{0}=\frac{\int_{M}Hd\mu}{\int_{M}X^{*}(e^{\Psi})d\mu}$.

Then
\begin{align*}
J'(0)=\int_{M} nf(-H+H_{0}e^{\psi}) d\mu,
\end{align*}
where $\psi$ denotes $x^{*}(\Psi)$.

\begin{prop}\label{critical}
The following three statements are equivalent:
\begin{enumerate}[(i)]
\item The mean curvature of $x(M)$ satisfies $H=Ce^{\psi}$ for a constant $C$.
\item For all weighted volume-preserving variations, $A'(0)=0$.
\item For arbitrary variations, $J'(0)=0$.
\end{enumerate}
\end{prop}

\begin{proof}
It is easy to check $(i)\Rightarrow (iii)$ and $(iii)\Rightarrow (ii)$. Now, we show $(ii)\Rightarrow (i)$. We can choose $f=-He^{-\psi}+H_{0}$ since $\int_{M}fe^{\psi}d\mu=0$. From $$0=J'(0)=\int_{M}n(-He^{-\psi}+H_{0})^{2}e^{\psi}d\mu,$$ we know $-He^{-\psi}+H_{0}=0$. Thus, we finish the proof.
\end{proof}

\end{document}